\documentclass[10pt]{article}
\usepackage{amsmath,amsthm,amssymb}
\usepackage{hyperref}

\usepackage{geometry, wrapfig}
\usepackage[matrix,arrow,curve,cmtip]{xy}
\usepackage{xcolor}
\usepackage{url}
\usepackage{filecontents}
\usepackage{tikz}
\usetikzlibrary{positioning}

\newcommand{\C}{\mathbb{C}}
\newcommand{\Q}{\mathbb{Q}}
\newcommand{\Z}{\mathbb{Z}}
\newcommand{\F}{\mathbb{F}}

\newcommand{\disc}{\mathrm{disc}}
\newcommand{\p}{\mathfrak{p}}
\newcommand{\q}{\mathfrak{q}}
\newcommand{\m}{\mathfrak{m}}
\newcommand{\Gal}{\mathrm{Gal}}
\newcommand{\Fix}{\mathrm{Fix}}

\newtheorem{thm}{Theorem}
\newtheorem{proposition}[thm]{Proposition}
\newtheorem{corollary}[thm]{Corollary}

\newtheorem{definition}[thm]{Definition}

\begin{document}
\title{Morphisms of Skew Hadamard Matrices}
\author{\textsc{Philip Heikoop}
	\thanks{\textit{E-mail: ptheikoop@wpi.edu}}\\
\textsc{Guillermo Nu{\~n}ez Ponasso}
\thanks{\textit{Email: gcnunez@wpi.edu}}\\
\textsc{Padraig \'{O} Cath\'{a}in}
\thanks{\textit{Email: pocathain@wpi.edu}}\\
\textsc{John Pugmire}
\thanks{\textit{Email: pugmire@ucsb.edu}} \\
\textit{\footnotesize{School of Mathematical Sciences}}\\
\textit{\footnotesize{Worcester Polytechnic Institute, MA 01609, USA}}\\
\date{\today}
}
\maketitle

\begin{abstract}
Quaternary unit Hadamard (QUH) matrices were introduced by Fender, Kharaghani and Suda along with a method to construct them at prime power orders. We present a novel construction of real Hadamard matrices from QUH matrices. Our construction recovers the result by Mukhopadhyay on  the existence of real Hadamard matrices of order $q^n+q^{n-1}$ for each prime power $q\equiv 3\mod 4$, and $n\geq 1$. Furthermore we provide nonexistence conditions for QUH matrices.
\end{abstract}

\section{Introduction}
A celebrated theorem of Hadamard characterises the complex matrices with entries of norm at most one which have maximal determinant: they are precisely the solutions to the matrix equation $HH^{\ast} = nI_{n}$ satisfying $|h_{ij}| = 1$ for all $1 \leq i,j\leq n$. Equivalently, all entries of $H$ have unit norm, and all rows are mutually orthogonal under the Hermitian inner product, \cite{Hadamard1893}. Real Hadamard matrices, having entries in $\{\pm1\}$, have been extensively studied for a century, though the existence problem is far from settled. We refer the reader to the recent monographs of Horadam and of de Launey and Flannery for extensive discussion of Hadamard matrices, \cite{HoradamHadamard, deLauneyFlannery}.

In this paper we will study the problem of constructing real Hadamard matrices from complex Hadamard matrices (CHM). Suppose that $X$ is a set of complex numbers of modulus $1$.
We define $\mathcal{H}(n, X)$ to be the set of $n \times n$ Hadamard matrices with entries drawn from $X$. In the special case that $X$ is the set of $k^{\textrm{th}}$ roots of unity, a CHM is called a \textit{Butson Hadamard matrix}; the set of such matrices is denoted $\mathcal{BH}(n,k)$. Examples of Butson Hadamard matrices are furnished by the character tables of abelian groups of order $n$ and exponent $k$. Cohn and Turyn proved independently that the existence of $H \in \mathcal{BH}(n,4)$ implies the existence of a real Hadamard matrix of order $2n$, \cite{Cohn65, Turyn}. More recently, Compton, Craigen and de Launey proved that an $n \times n$ matrix with entries in the \textit{unreal} sixth roots of unity $\{\omega_{6}, \omega_{6}^{2}, \omega_{6}^{4}, \omega_{6}^{5}\}$ can be used to construct a real Hadamard matrix of order $4n$, \cite{CCdeL}.

A general construction for mappings between sets of Butson Hadamard matrices is described by Egan and one of the present authors, \cite{mypaper-morphisms}. A key ingredient in the construction is a matrix $H \in \mathcal{BH}(n, k)$ with minimal polynomial $\Phi_{t}(x)$ for some integer $t$. The construction of such matrices was considered further in collaboration with
Eric Swartz, \cite{mypaper-explicitmorphisms}. In all the examples considered previously, matrix entries are roots of unity, and all fields considered are cyclotomic. In this paper, we consider a family of complex Hadamard matrices with entries in the biquadratic extension $\mathbb{Q}[\sqrt{-q}, \sqrt{q+1}]$. When the matrix entries are all in the set $X_{q} = \{ \frac{\pm 1\pm \sqrt{-q}}{\sqrt{q+1}} \}$, such a matrix is called a \textit{Quaternary Unit Hadamard matrix}, abbreviated QUH. Such matrices were first considered by Fender, Kharaghani and Suda, \cite{FKS}.

We will construct a morphism from QUH matrices onto real Hadamard matrices, using skew-Hadamard matrices. This provides a new construction for a family of Hadamard matrices of order $q^{n} + q^{n-1}$ for each prime power $q \equiv 3 \mod 4$ and each $n\geq 1$, previously constructed by Mukhopadhyay and studied further by Seberry, \cite{Mukh, SeberrySkew}. We conclude the paper by studying the decomposition of prime ideals in the field $\mathbb{Q}[\sqrt{-q}, \sqrt{q+1}]$ to obtain non-existence results for QUH matrices in the style of Winterhof \cite{WinterhofExistence}.

\section{Morphisms of QUH matrices}

In this section we construct an isomorphism of fields, which we lift to an isomorphism of matrix algebras. We prove that this isomorphism carries a QUH matrix in the set $\mathcal{H}(n, X_{m})$ to a real Hadamard matrix of order $n(m+1)$; that is, the isomorphism is a \textit{morphism} of complex Hadamard matrices. We will require some standard results in algebra, as discussed in, e.g., Chapters 17--19 of Isaacs' \textit{Graduate Algebra}, \cite{Isaacs}. An \textit{extension field} $k$ of $\Q$ is a field containing $\Q$ as a subfield; in this case $k$ is a $\Q$-vector space and its \textit{degree} is its dimension as a vector space. The degree of $k$ over $\Q$ is denoted by $[k:\Q]$.  In the ring $\mathbb{Q}[x]$ every ideal contains a unique monic polynomial of minimal degree, this polynomial is irreducible if and only if the ideal is maximal. For a polynomial $p(x)$ the quotient $\mathbb{Q}[x]/\left(p(x) \right)$ is a field if and only if the polynomial $p(x)$ is irreducible. An extension field $k$ is the \textit{splitting field} of a polynomial $p(x) \in \mathbb{Q}[x]$ if $k$ is a field of minimal degree over $\mathbb{Q}$ which contains all the roots of $p(x)$. We apply these results to the polynomial $\mathfrak{m}(x) = x^{4} + \frac{2(m-1)}{m+1}x^2+1$. By abuse of notation, a Hadamard matrix is \textit{skew} if $H-I$ is a skew-symmetric matrix.

\begin{proposition}\label{fieldprop}
Let $H$ be a skew-Hadamard matrix of order $m+1$, where $m+1$ is not a perfect square.
The minimal polynomial of $\alpha_{m} = \frac{1+\sqrt{-m}}{\sqrt{m+1}}$ and the minimal
polynomial of $\frac{1}{\sqrt{m+1}}H$ are both equal to
\[ \mathfrak{m}(x) = x^4+\frac{2(m-1)}{m+1}x^2+1 \,.\]
\end{proposition}

\begin{proof}
It is easily checked that $\alpha_m$ is a root of $\m(x)$. Since $\m(x)$ is even, $-\alpha_{m}$ is also a root. The coefficients of $\m(x)$ are real, thus $\alpha_{m}^{\ast}$ and $-\alpha_{m}^{\ast}$ are roots. From the fact that $\m(x)$ has degree $4$, we conclude that these are all the possible roots. Therefore we obtain the factorisation \[\m(x)=(x-\alpha_m)(x-\alpha_m^{\ast})(x+\alpha_m)(x+\alpha_m^{\ast})\,.\]
Clearly $\m(x)$ has no linear factors in $\Q[x]$. The only possible quadratic factors with real entries are $(x-\alpha_m)(x-\alpha_m^{\ast})=x^2-2x/\sqrt{m+1}+1$ and $(x+\alpha_m)(x+\alpha_m^{\ast})=x^2+2x/\sqrt{m+1}+1$. By hypothesis, $m+1$ is not a perfect square so these factors are not in $\Q[x]$. We have shown that $\m(x)$ is irreducible.
The field extension $\Q[\alpha_{m}]$ contains $\alpha^{-1} = \alpha_{m}^{\ast}$ and $-\alpha_{m}$, so it is the splitting field of $\m(x)$.

Since $H$ is skew-Hadamard we have both $HH^{\top} = (m+1)I_{m+1}$ and $H^{\top} = 2I - H$. It follows that $H(2I-H) = (m+1)I$, or $H^{2} = 2H - (m+1)I$.
Hence,
\[ \left(\frac{1}{\sqrt{m+1}} H\right)^{2} = \frac{2}{m+1} H - I \,.\]
We also compute that
\begin{eqnarray*}
 \left(\frac{1}{\sqrt{m+1}} H\right)^{4} & = & \frac{4}{(m+1)} \left(\frac{1}{\sqrt{m+1}} H\right)^{2} - \frac{4}{m+1}H + I \\
 & = &  \frac{4}{(m+1)} \left(\frac{1}{\sqrt{m+1}} H\right)^{2} - 2 \left( \frac{2}{m+1}H - I \right) - I \\
 & = &  \frac{4}{(m+1)} \left(\frac{1}{\sqrt{m+1}} H\right)^{2} - 2 \left(\frac{1}{\sqrt{m+1}} H\right)^{2} - I \\
 & = & \frac{2-2m}{m+1} \left(\frac{1}{\sqrt{m+1}} H\right)^{2} - I
\end{eqnarray*}
We conclude that the unitary matrix $\frac{1}{\sqrt{m+1}}H$ is a root of polynomial $\mathfrak{m}(x)$, which must be the minimal polynomial of $\frac{1}{\sqrt{m+1}}H$ by irreducibility.
\end{proof}

When $m+1$ is a perfect square, the polynomial $\m(x)$ factors into two irreducible quadratic factors in $\Q[x]$, which correspond to the distinct minimal polynomials of $\alpha_m$ and $-\alpha_m$. In this case, the minimal polynomials of $\alpha_m$ and $\frac{1}{\sqrt{m+1}}H$ coincide, and also the minimal polynomials of $-\alpha_m$ and $\frac{1}{\sqrt{m+1}}(H-2I)$ coincide. The case that $m+1$ is a perfect square will be discussed after the proof of Theorem \ref{skewmorphism}. From Proposition \ref{fieldprop}, the next result is immediate.

\begin{proposition}\label{isomorphism} If $H$ is a skew-Hadamard matrix of order $m$, then all of the following $\Q$-algebras are isomorphic:
  \begin{equation}\label{isomorphismEq}
  	\Q[x]/(\mathfrak{m}(x))\simeq \Q\left[\frac{1}{\sqrt{m+1}}H\right] \simeq \Q[\alpha_{m}].
  \end{equation}
\end{proposition}

\begin{definition}
A \textit{Quaternary Unit Hadamard} (QUH) matrix is an element of $\mathcal{H}(n,X_m)$, where
\[ X_{m} = \left\{ \frac{\pm 1 \pm \sqrt{-m} } {\sqrt{m+1}} \right\}\,. \]
\end{definition}

Now we give the main result of this section.

\begin{thm}\label{skewmorphism}
If there exists a skew-Hadamard matrix $H$ of order $m+1$, where $m+1$ is not a perfect square, there exists a morphism $\mathcal{H}(n,X_m)\rightarrow BH(nm+n,2)$.
\end{thm}

\begin{proof}
Assume that there exists $M\in \mathcal{H}(n,X_m)$, since otherwise the claim is vacuous.
By Proposition \ref{isomorphism} that there exists an isomorphism $\mathbb{Q}(\alpha_{m}) \rightarrow \mathbb{Q}(\frac{1}{\sqrt{m+1}}H)$. We make this explicit:
\begin{equation*}
	\varphi: \alpha_m\mapsto \frac{1}{\sqrt{m+1}}H
\end{equation*}
and since $\alpha_m$ is a generator of $\mathbb{Q}(\alpha_{m})$ the function $\varphi$ extends uniquely to the whole field. Recalling that $H$ is skew, we obtain
\[ \varphi(-\alpha_m)=\frac{-1}{\sqrt{m+1}}H = \frac{1}{\sqrt{m+1}}(H-2I)^{\top},\,\,\, \varphi(\alpha_{m}^{\ast}) = \frac{1}{\sqrt{m+1}}H^{\top}\,. \]

Define $M^\varphi$ to be the block matrix obtained from $M$ by applying $\varphi$ entrywise. Then $M^{\varphi}$ is a real matrix  of size $n(m+1)\times n(m+1)$ with entries in the set $\{\pm 1/\sqrt{m+1}\}$. Since $M\in \mathcal{H}(n,X_m)$ the (Hermitian) inner product of columns $c_{i}$ and $c_{j}$ of $M$ is $\langle c_i, c_j\rangle = n\delta_{i}^{j}$, where $\delta_{i}^{j}$ is the Kronecker $\delta$ function. Since $\varphi$ is an isomorphism of $\Q$-algebras, $\varphi(0)=\mathbf{0}_{m+1}$ and $\varphi(1)= I_{m+1}$. It follows that
  \begin{align*}
	\sum_{k} \varphi(c_{i,k})\varphi(c_{j,k})^{\top} &= \sum_{k} \varphi(c_{i,k})\varphi(c_{j,k}^{\ast})\\
	&= \varphi\left(\sum_k c_{i,k}c_{j,k}^{\ast}\right)\\
	&= \varphi(\langle c_i, c_j\rangle)\\
	&= n\delta_{i}^{j}I_{m+1}\,.
  \end{align*}
This shows that $M^{\varphi}\left(M^{\varphi}\right)^{\top} = nI_{n(m+1)}$. The entries of $M^{\varphi}$ are in the set $\{\pm 1/\sqrt{m+1}\}$, so the entries of $\sqrt{m+1}M^{\varphi}$ are in the set $\{\pm 1\}$. We have shown that
  \begin{equation*}
	\sqrt{m+1}M^{\varphi}\left(\sqrt{m+1}M^{\varphi}\right)^\top = n(m+1)I_{n(m+1)} \,,
  \end{equation*}
which establishes the theorem.
\end{proof}

A less technical method to prove the above theorem without assumptions on $m+1$ is as follows: Let $H\in\mathcal{H}(n, X_m)$, and let 
\[H = \frac{1}{\sqrt{m+1}}A + \frac{\sqrt{-m}}{\sqrt{m+1}}B,\]
where $A$ and $B$ are $\pm 1$ matrices of order $n$. Then 
\[AB^{\top} = BA^{\top}\hbox{ and } AA^{\top}+BB^{\top} = n(m+1)I_n.\]

Let $M$ be a skew Hadamard matrix of order $m+1$. Substituting $A$ for the diagonal entries of $M$ and $\pm B$ for the off-diagonal entries $\pm 1$ of $M$, it can be verified that the resulting matrix will be a Hadamard matrix of order $n(m+1)$. Although this proof is simpler than that of Theorem 4, our method gives additional insights into existence and non-existence of QUH matrices, as demonstrated in Section 3.\\

Let $q$ be an odd prime power and $\F_q$ be a finite field with $q$ elements. The element $a \in \mathbb{F}_{q}$ is a \textit{quadratic residue} if there exists $y \in \mathbb{F}_{q}$ such that $y^{2} = a$. Otherwise, $a$ is a non-residue.  The \textit{quadratic character} is defined to be $\chi_q(a)=1$ if $a \in \F_q^*=\F_q-\{0\}$ is a quadratic residue in $\F_q$, $\chi_q(a)=-1$ if $a\in \F_q^*$ is a quadratic non-residue in $\F_q$ and $\chi_q(0)=0$. In the case where $q=p$ is a prime number, the quadratic character $\chi_p(a)$ on $\F_p\simeq \Z/p\Z$ can be identified with the \textit{Legendre symbol} and is denoted $(a/p)$. Later we will use the fact that for a fixed prime $p$ and for every $a,b\in\Z$, $\left(ab/{p}\right)=\left({a}/{p}\right)\left({b}/{p}\right)$, \cite[Proposition 5.1.2]{IrelandRosen}. Let $\{g_0=0,g_1,\dots,g_{q-1}\}$ be an enumeration of $\F_q$ then $Q = \left[ \chi_q(g_i-g_j)\right]_{0\leq i,j\leq q-1}$ is the \textit{Jacobsthal matrix} of order $q$.

\begin{thm}[Section 3, \cite{FKS}]
Let $q$ be an odd prime power with $q\equiv 3\pmod 4$. Define $1\times 1$ matrices $A_{0} = B_{0} = 1$, let $Q$ be the $q \times q$ Jacobsthal matrix and $J_q$ the $q\times q$ all-ones matrix. For each $t\geq 1$, define
\[ A_{t} = J_{q} \otimes B_{t-1}, \,\,\, B_{t} = I_{q} \otimes A_{t-1} + Q \otimes B_{t-1} \,.\]
Then for each $t$ the matrix $\frac{1}{\sqrt{q+1}}A_{t} + i \frac{\sqrt{q}}{\sqrt{q+1}} B_{t}$ is a matrix in $\mathcal{H}(q^t,X_{q})$.
\end{thm}

Hence there exist $\mathcal{H}(q^{t}, X_q)$ matrices for all prime powers $q \equiv 3 \mod 4$. Since the Paley matrix of order $q + 1$ is skew, we can apply Theorem \ref{skewmorphism} to obtain the following result.

\begin{corollary}
Let $q \equiv 3 \mod 4$ be a prime power. For any integer $n \geq 1$ there exists a (real) Hadamard matrix of order $q^{n} + q^{n-1}$.
\end{corollary}

This result was first discovered by Mukhopadhyay, and later clarified and elaborated by Seberry, \cite{Mukh, SeberrySkew}. Of course, it would be interesting to develop constructions of Hadamard matrices at previously unknown orders. As a first contribution in this direction, we investigate the non-existence of QUH matrices in the next section.

\section{Nonexistence of quaternary unit Hadamard matrices}

The \textit{Galois group} of an irreducible polynomial $p(x)$ is the group of field automorphisms of a splitting field of $p(x)$. Over $\Q$, the order of the Galois group and the degree of the splitting field coincide. The Galois correspondence gives an inclusion-reversing bijection between the lattice of subfields of $\Q[x]/\left(p(x)\right)$ and the subgroups of the Galois group.

An element $x\in \C$ is an \textit{algebraic integer} if it is a root of a monic polynomial in $\Z[x]$. The ring of integers of a number field $k\subseteq\mathbb{C}$ is the largest subring of the algebraic integers contained in $k$, usually denoted $\mathcal{O}_k$. In the ring of integers of a number field, ideals factorise uniquely as a product of \textit{prime ideals}, \cite[Theorem 14]{Marcus}. Studying prime factorisations related to the determinant of a putative complex Hadamard matrix can sometimes yield nonexistence results. This argument is similar to one given by Winterhof for certain Butson Hadamard matrices, \cite{WinterhofExistence}.

First, we introduce terminology for the factorisation of a prime ideal of $\mathbb{Z}$ in $\mathcal{O}_{k}$ for a number field $k$. As is customary we will denote prime ideals in $k$ by the gothic letters $\p$ and $\q$ and rational primes by $p$ and $q$.

\begin{definition} Let $k$ be the splitting field of an irreducible polynomial, and $q$ be a rational prime.
\begin{itemize}
\item $q$ is \textit{inert} in $\mathcal{O}_k$ if $(q)$ is a prime ideal in $\mathcal{O}_k$.
\item If $q$ is not inert then it \textit{splits} in $\mathcal{O}_k$. Let $(q)=\prod \q_i^e$ be the prime ideal decomposition of $(q)$. If $e\geq 1$ then $q$ is \textit{ramified}, otherwise it \textit{splits completely}.
\end{itemize}
\end{definition}

The \textit{discriminant} of a number field is an integer valued invariant that controls the factorisation of rational primes in that field. The following result is a special case of a more general result on the splitting of rational primes on number fields, see Theorems 21, 23 and 24 of Marcus' \textit{Number Fields} for details, \cite{Marcus}.

\begin{thm}\label{splittingthm}
Let $k$ be a number field. If a rational prime $q$ is ramified in $\mathcal{O}_k$, then $q\mid \disc(k)$. Let $k$ be the splitting field of some irreducible polynomial, where the degree of $k$ over $\Q$ is $n=[k:\Q]$. If $q$ is a rational prime such that $q\nmid \disc(k)$, then
\[(q)=\q_1\dots \q_r,\]
where $r|n$. Furthermore the action of the Galois group on $\{\q_1,\dots,\q_r\}$ is transitive.
\end{thm}

In a quadratic extension of $\mathbb{Q}$, the Legendre symbol controls the splitting of prime ideals.

\begin{thm}[p.24, Theorem 25, \cite{Marcus}]\label{quadsplit}
Let $k=\Q[\sqrt{d}]$ where $d$ is a squarefree integer. Then $\disc(k)= d$ if $d \equiv 1 \mod 4$ and $\disc(k) = 4d$ if $d \equiv 2, 3 \mod 4$. Suppose that $q$ is an odd rational prime and $q\nmid \disc(k)$. Then
\begin{itemize}
\item $q$ is inert in $\mathcal{O}_k$ if $\left({d}/{q}\right)=-1$.
\item $q$ splits into distinct prime ideals in $\mathcal{O}_k$ if $\left({d}/{q}\right)=1$.
\end{itemize}
\end{thm}

We will study these concepts for the field $K=\Q[\alpha]$, which by Proposition \ref{fieldprop} is the splitting field of $\m(x)$. Since $2/(\alpha_m+\alpha^{\ast}_m)=\sqrt{m+1}$ and $(\sqrt{m+1})\alpha_m -1= \sqrt{-m}$ we have an isomorphism $\Q[\alpha_m]\simeq \Q[\sqrt{-m},\sqrt{m+1}]$. There are three intermediate subfields of $K$, as illustrated.

\begin{center}
\begin{tikzpicture}[node distance =2cm]
\node(K) {$K=\Q\left[\sqrt{m+1},\sqrt{-m}\right]$};
\node(K2)[below=.75cm of K]{$K_2=\Q\left[\sqrt{m+1}\right]$};
\node(K1)[left=1.25cm of K2] {$K_1=\Q\left[\sqrt{-m}\right]$};
\node(K3)[right=1.25cm of K2]{$K_3=\Q\left[\sqrt{-m(m+1)}\right]$};
\node(Q)[below =2.25cm of K]{$\Q$};
\foreach \x in {1,2,3}{
	\draw(K\x)--(K);
	\draw(Q)--(K\x);
}
\end{tikzpicture}
\end{center}
\begin{center}
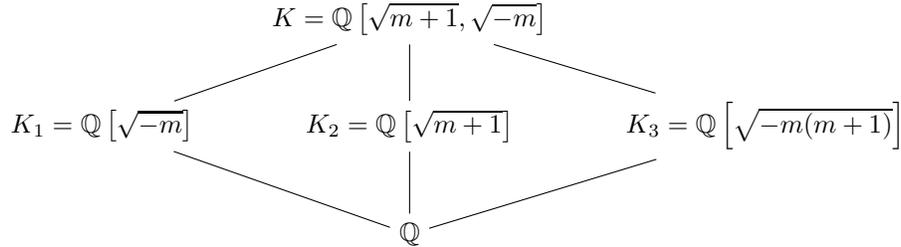
\textit{The lattice of subfields of $K$.}
\end{center}

The discriminant of a biquadratic extension is given as an exercise by Marcus.

\begin{proposition}[p.36-37, \cite{Marcus}]\label{discprop}
The discriminant of a biquadratic extension $k=\Q[\sqrt{a},\sqrt{b}]$ where $\gcd(a,b)=1$ is
\[\disc(k)=\disc(k_1)\disc(k_2)\disc(k_3),\]
where $k_1=\Q[\sqrt{a}]$, $k_2=\Q[\sqrt{b}]$ and $k_3=\Q[\sqrt{ab}]$.
\end{proposition}

Let $G=\Gal(K/\Q)$ be the Galois group the splitting field of $\m(x)$. By the Galois correspondence $G$ has order 4, and has three distinct subgroups of order 2. So $G$ is elementary abelian, generated by $\sigma:\sqrt{m+1}\mapsto -\sqrt{m+1}$ and $\tau:\sqrt{-m}\mapsto -\sqrt{-m}$. We identify $\tau$ with complex conjugation. Note that $K_1=\Fix(\sigma)$ is the fixed field of $\sigma$, that $K_2=\Fix(\tau)$ is the fixed field of $\tau$ and $K_3=\Fix(\sigma\tau)$ is the fixed field of $\sigma\tau$.

From now on, let $m = p$ be a prime congruent to $3$ modulo $4$, and write $s$ for the squarefree part of $p+1$. Then $K\simeq\Q[\sqrt{-p},\sqrt{s}]$, and applying Proposition \ref{discprop} we have
\[\disc(K)=
\begin{cases}
s^2p^2 &\hbox{ if } s\equiv 1\mod 4\\
16s^2p^2 & \hbox{ if } s\equiv 2,3 \mod 4
\end{cases}.
\]

Let $q$ be a prime number. By Theorem \ref{splittingthm}, the prime $q$ ramifies in $\mathcal{O}_K$ only if $q=p$ or $q|s$. Next we describe which non-ramified primes split in $\mathcal{O}_{K}$.

\begin{proposition}\label{biquadsplit}
Let $q$ be a rational prime not dividing $\disc(k)$. Then one of the following holds:
\begin{itemize}
\item $(q)=\q_1\q_2\q_3\q_4$ in $\mathcal{O}_K$ and $q$ splits in every subfield of $K$.
\item $(q)=\q_1\q_2$ in $\mathcal{O}_K$ and $q$ splits in one proper subfield of $K$, being inert in the other two.
\end{itemize}
\end{proposition}
\begin{proof}
By Theorem \ref{quadsplit}, the prime $q$ splits in $K_{1}$ if and only if $\left({-p}/{q}\right) = 1$, and $q$ splits in $K_{2}$ if and only if $\left({s}/{q}\right) = 1$.
Suppose that $\left({-p}/{q}\right)=\left({s}/{q}\right)=-1$. Then $\left({-ps}/{q}\right)=\left({-p}/{q}\right)\left({s}/{q}\right)=1$, so $q$ splits in $K_3$.
We conclude that no rational prime is inert in $K$.

Since by assumption $q$ does not ramify, Theorem \ref{splittingthm} tells us that $q$ splits in $\mathcal{O}_K$ into two or four prime ideals. Suppose that $(q)=\q_1\q_2\q_3\q_4$. Then up to a relabeling of the primes $\q_i$ we can assume that
\begin{center}
\begin{tabular}{l l}
$\q_1^{\sigma}=\q_2,$ & $\q_3^{\sigma}=\q_4$\\
$\q_1^{\tau}=\q_3,$ & $\q_2^{\tau}=\q_4$\\
$\q_1^{\sigma\tau}=\q_4,$ & $\q_2^{\sigma\tau}=\q_3$\\
\end{tabular}
\end{center}
This implies that $(\q_1\q_2)^{\sigma}=\q_1\q_2$ and $(\q_3\q_4)^{\sigma}=\q_3\q_4$, therefore $\q_1\q_2$ and $\q_3\q_4$ are ideals in the fixed field  $K_1$ of $\sigma$ and thus $q$ splits as $(q)=(\q_1\q_2)(\q_3\q_4)$ in $K_1$. We can show analogously that $q$ splits in $K_2$ and $K_3$. Suppose next that $q$ splits in $K$ as $\q_1\q_2$. Then the Galois group acts as in one of the following possibilities.
\begin{center}
\begin{tabular}{|c|c|c|c|}
\hline
 $\q_1^{\sigma}$ & $\q_1^{\tau}$ & $\q_1^{\sigma\tau}$ & Subfield containing $\q_1$ and $\q_2$\\
\hline
 $\q_1$ & $\q_2$ & $\q_2$ & $K_1=\Fix(\sigma)$\\
 $\q_2$ & $\q_1$ & $\q_2$ & $K_2=\Fix(\tau)$\\
 $\q_2$ & $\q_2$ & $\q_1$ & $K_3=\Fix(\sigma\tau)$\\
\hline
\end{tabular}
\end{center}
In each case, there is exactly one non-identity element $g\in G$ fixing both $\q_1$ and $\q_2$. So $q$ splits in the fixed field of $g$, and is inert in the other two intermediate subfields. 
\end{proof}

In our application to QUH matrices, we will require the following special case of Proposition \ref{biquadsplit}.
\begin{corollary}\label{splitcor}
Let $q$ be an odd rational prime $q$, coprime to both $p$ and $s$. In $\mathcal{O}_K$, we have $(q)=\q_1\q_2$ with $\q_1^\tau=\q_1$ and $\q_2^\tau=\q_2$ if and only if $(-p/q)=-1$ and $(s/q)=1$.
\end{corollary}

\begin{proof}
Since $\q_1^\tau=\q_1$ it must be the case that $\q_{1}^{\sigma} = \q_{2}$ and, by Proposition \ref{biquadsplit}, $q$ splits in $K_{2}$ as $\q_{1}\q_{2}$. So by Theorem \ref{quadsplit}, we must have $(s/q) = 1$. Furthermore, $(q)$ must be inert in $K_{1}$, from which we obtain $(-p/q)= -1$ as required. The converse follows from Theorem \ref{quadsplit} and Proposition \ref{biquadsplit}.
\end{proof}

Recall that the action of $\tau$ on $K$ corresponds to the action of complex conjugation on $K$. Therefore the case above is equivalent to $(q)=\q_1\q_2$ with $\q_1^*=\q_1$ and $\q_2^*=\q_2$. We can now formulate our main nonexistence theorem.

\begin{thm}\label{nonexistence}
Let $n$ be an odd integer, with squarefree part $t$. Let $p\equiv 3\mod 4$ be a prime number such that the squarefree part of $p+1$ is $s>1$. If there exists an odd prime $q$ such that
\begin{itemize}
\item $q$ divides $t$,
\item $\left({s}/{q}\right)=1$, and
\item $\left({-p}/{q}\right)=-1$,
\end{itemize}
then $\mathcal{H}(n,X_p)$ is empty.
\end{thm}
\begin{proof}
Let $M\in \mathcal{H}(n,X_p)$ and set $D=(p+1)^n\det M$. Then $ D \in \mathcal{O}_K$, since $(p+1)\alpha\in\mathcal{O}_K$ for every $\alpha\in X_p$. The matrix $H$ is complex Hadamard, therefore $DD^*=(p+1)^{2n}n^n=a^2t^n$, for some $a\in \Z$. By Corollary \ref{splitcor}, $(q)=\q_1\q_2$ in $\mathcal{O}_K$ with  $\q_1=\q_1^*$. We have that $q|t$, so since $n$ is odd the prime ideal $\q_1$ appears with odd multiplicity in the decomposition of $(p+1)^{2n}n^n$ in $\mathcal{O}_K$.
Since $\q_{1}$ is prime and divides the product $(D)(D^{\ast})$, it divides one of the factors; without loss of generality, suppose that $\q_1$ divides $(D)$. So $(D)$ factors into prime ideals uniquely as
	\[(D)=\q_1^{\ell}\prod_j \p_j^{\ell_j},\]
Then $(D^*)=(D)^*=\q_1^{\ell}\prod_j(\p_j^*)^{\ell_j}$. But implies that $\q_1$ appears with even multiplicity in $(D)(D^{\ast})$, contradicting its odd multiplicity in $(p+1)^{2n}n^{n}$.
\end{proof}

The only prime of the form $n^2-1$ is 3. In this case the matrices $\mathcal{H}(n,X_3)$ coincide with the unreal $BH(n,6)$ matrices of Compton, Craigen and de Launey. The set $\mathcal{H}(n, X_{3})$ is empty if there exists a prime $q \equiv 5 \pmod {6}$ which divides the square-free part of $n$ (see Theorem 2 of \cite{CCdeL} or Theorem 5 of \cite{WinterhofExistence} for a proof).\\

We conclude this paper by discussing some consequences of Theorem \ref{nonexistence}. Suppose first that $p=7$. Then a prime $q$ satisfying both $(q/7) = -1$ and $(2/q) = 1$ cannot divide the square-free part of $n$. By quadratic reciprocity, these are the primes which satisfy both $q \equiv 3, 5, 6 \mod 7$ and $q \equiv 1, 7 \mod 8$. By Dirichlet's Theorem on primes in arithmetic progressions, there are infinitely many such primes. Similar results hold for each prime $p$, as illustrated in the table below.\\

\begin{center}
\begin{tabular}{|l |l |}
\hline
$p$ & $n$ \\
\hline
 $7$ & $17,31,41,47,51,73,85,89,93,97,103,119,123,141,\dots$\\
 $11$ & $13,39,61,65,73,83,91,107,109,117,131,143,167,\dots$\\
 $19$ & $29,31,41,59,71,79,87,89,93,109,123,145,151,\dots$\\
 $23$ & $5,15,19,35,43,45,53,55,57,65,67,85,95,97,105,\dots$\\
 $31$ & $17,23,51,69,73,79,85,89,115,119,127,137,151,\dots$\\
 $43$ & $5,7,15,19,21,35,37,45,55,57,63,65,77,85,89,91,\dots$\\
 \hline
\end{tabular}\\\vspace{0.5cm}

\textit{Pairs $(n,p)$ such that $\mathcal{H}(n,p)$ is empty.}
\end{center}

In fact, it is a consequence of the Chebotarev Density Theorem that the proportion of primes $q\leq N$ to which the conditions of Theorem \ref{nonexistence} apply tends to $1/4$ as $N$ tends to infinity. In particular, there are infinitely many primes which obstruct the existence of matrices in $\mathcal{H}(n, X_{p})$ for any fixed $p$.\\

To illustrate Theorem \ref{nonexistence} in a case where not all ideals are principal, we consider $p=43$ and $q=5$, then $s=11$. We have $(5/43)=-1$, thus the prime $5$ should be inert in $\mathcal{O}_{K_1}$.
By Proposition \ref{biquadsplit}, $(5)$ splits in $\mathcal{O}_K$ as the product of two prime ideals in $\mathcal{O}_{K_2}$, indeed  $(5)=(5,1+\sqrt{11})(5,1-\sqrt{11})$ in $\mathcal{O}_K$. If there exists $H\in\mathcal{H}(5,X_{43})$ then $D=11^5\det H$ and $DD^*=11^{10}\cdot 5^5$. Thus in $\mathcal{O}_K$ this means
\[(D)(D)^* = (11^{5})^2 (5,1+\sqrt{11})^{5}(5,1-\sqrt{11})^{5}.\]
The ideal $(5,1+\sqrt{11})=(5,1+\sqrt{11})^*$ appears with even multiplicity on the left hand side and odd multiplicity on the right hand side. Hence $\mathcal{H}(5,X_{43})$ is empty.

\section*{Acknowledgements}

This research was completed while JP and PH were undergraduates and GNP was a doctoral student at Worcester Polytechnic Institute.
JP was supported by a Student Undergraduate Research Fellowship sponsored by the office of the Dean of Arts and Sciences. PH and GNP
were supported by P\'{O}C's startup funds.
The authors acknowledge the anonymous referees for many helpful suggestions which improved the exposition of the paper.

\bibliographystyle{abbrv}
\flushleft{
\bibliography{HeikoopPugmirePonassoOCathain}
}

\end{document}